\documentclass[12pt,reqno]{amsart}

\usepackage{amsmath,amsthm,amssymb,comment,fullpage}
\usepackage{times}
\usepackage[T1]{fontenc}
\usepackage{mathrsfs}
\usepackage{latexsym}
\usepackage[dvips]{graphics}
\usepackage{epsfig}
\usepackage{amsmath,amsfonts,amsthm,amssymb,amscd}
\input amssym.def
\input amssym.tex
\usepackage{color}
\usepackage{hyperref}
\usepackage{url}

\newcommand{\prob}[1]{{\rm Prob}\left(#1\right)}
\newcommand{\p}[1]{{\rm Prob}\left(#1\right)}
\newcommand{\E}{\mathbb{E}}

\newcommand{\burl}[1]{\textcolor{blue}{\url{#1}}}

\numberwithin{equation}{section}
\newtheorem{thm}{Theorem}[section]

\newtheorem{cor}[thm]{Corollary}
\newtheorem{lem}[thm]{Lemma}

\theoremstyle{plain}

\newtheorem{defn}[thm]{Definition}

\newcommand\be{\begin{equation}}
\newcommand\ee{\end{equation}}
\newcommand\bea{\begin{eqnarray}}
\newcommand\eea{\end{eqnarray}}
\newcommand\bi{\begin{itemize}}
\newcommand\ei{\end{itemize}}
\newcommand\ben{\begin{enumerate}}
\newcommand\een{\end{enumerate}}
\newcommand\bc{\begin{center}}
\newcommand\ec{\end{center}}
\newcommand\ba{\begin{array}}
\newcommand\ea{\end{array}}

\newcommand{\e}{\varepsilon}

\newcommand{\Z}{\ensuremath{\mathbb{Z}}}

\newcommand{\N}{\mathbb{N}}

\newcommand\frakfamily{\usefont{U}{yfrak}{m}{n}}
\DeclareTextFontCommand{\textfrak}{\frakfamily}

\newcommand{\twocase}[5]{#1 \begin{cases} #2 & \text{{\rm #3}}\\ #4 &\text{{\rm #5}} \end{cases}}
\newcommand{\threecase}[7]{#1 \begin{cases} #2 & \text{{\rm #3}}\\ #4 &\text{{\rm #5}}\\ #6 & \texttt{{\rm #7}} \end{cases}}

\newcommand{\hr}[1]{\href{#1}{\url{#1}}}

\title{Benford Behavior of Zeckendorf Decompositions}

\author{Andrew Best}
\email{\textcolor{blue}{\href{mailto:ajb5@williams.edu)}{ajb5@williams.edu}}}
\address{Department of Mathematics and Statistics, Williams College, Williamstown, MA 01267}

\author{Patrick Dynes}
\email{\textcolor{blue}{\href{mailto:pdynes@clemson.edu)}{pdynes@clemson.edu}}}
\address{Department of Mathematical Sciences, Clemson University, Clemson, SC 29634}

\author{Xixi Edelsbrunner}
\email{\textcolor{blue}{\href{mailto:xe1@williams.edu}{xe1@williams.edu}}}
\address{Department of Mathematics and Statistics, Williams College, Williamstown, MA 01267}

\author{Brian McDonald}
\email{\textcolor{blue}{\href{mailto:bmcdon11@u.rochester.edu}{bmcdon11@u.rochester.edu}}}
\address{Department of Mathematics, University of Rochester, Rochester, NY 14627}

\author{Steven J. Miller}
\email{\textcolor{blue}{\href{mailto:sjm1@williams.edu, Steven.Miller.MC.96@aya.yale.edu}{sjm1@williams.edu,Steven.Miller.MC.96@aya.yale.edu} }}
\address{Department of Mathematics and Statistics, Williams College, Williamstown, MA 01267}

\author{Kimsy Tor}
\email{\textcolor{blue}{\href{mailto:ktor.student@manhattan.edu}{ktor.student@manhattan.edu}}}
\address{Department of Mathematics, Manhattan College, Riverdale, NY 10471}

\author{Caroline Turnage-Butterbaugh}
\email{\textcolor{blue}{\href{mailto:cturnagebutterbaugh@gmail.com}{cturnagebutterbaugh@gmail.com}}}
\address{Department of Mathematics, North Dakota State University ,Fargo, ND 58102}

\author{Madeleine Weinstein}
\email{\textcolor{blue}{\href{mailto:mweinstein@g.hmc.edu}{mweinstein@g.hmc.edu}}}
\address{Department of Mathematics, Harvey Mudd College, Claremont, CA 91711 }

\thanks{This research was conducted as part of the 2014 SMALL REU program at Williams College and was supported by NSF grants DMS 1347804 and DMS 1265673, Williams College, and the Clare Boothe Luce Program of the Henry Luce Foundation. It is a pleasure to thank them for  their support, and the participants there and at the 16\textsuperscript{th} International Conference on Fibonacci Numbers and their Applications for helpful discussions. }



\date{\today}

\begin{document}

\begin{abstract}
A beautiful theorem of Zeckendorf states that every integer can be written uniquely as the sum of non-consecutive Fibonacci numbers $\{ F_i \}_{i = 1}^{\infty}$. A set $S \subset \Z$ is said to satisfy Benford's law if the density of the elements in $S$ with leading digit $d$ is $\log_{10}{(1+\frac{1}{d})}$; in other words, smaller leading digits are more likely to occur. We prove that, as $n\to\infty$, for a randomly selected integer $m$ in $[0, F_{n+1})$ the distribution of the leading digits of the Fibonacci summands in its Zeckendorf decomposition converge to Benford's law almost surely. Our results hold more generally, and instead of looking at the distribution of leading digits one obtains similar theorems concerning how often values in sets with density are attained.
\end{abstract}

\maketitle
\tableofcontents


\section{Introduction}


\subsection{History}\hspace*{\fill} \\

The Fibonacci numbers have fascinated professional mathematicians and amateurs for centuries. The purpose of this article is to review the connection between two interesting results, namely Zeckendorf's theorem and Benford's law of digit bias, and to discuss density results that arise in special subsets of the Fibonacci numbers.

A beautiful theorem due to Zeckendorf \cite{Ze} states that every positive integer may be written uniquely as a sum of non-adjacent Fibonacci numbers. The standard proof is by straightforward induction and the greedy algorithm (though see \cite{KKMW} for a combinatorial approach). For this theorem to hold we must normalize the Fibonacci numbers by taking $F_1 = 1$ and $F_2 = 2$ (and of course $F_{n + 1} = F_{n} + F_{n - 1}$), for if our series began with two $1$'s or with a 0 the decompositions of many numbers into non-adjacent summands would not be unique.

In 1937 the physicist Frank Benford \cite{Benf}, then working for General Electric, observed that the distributions of the leading digits of numbers in many real and mathematical data sets were not uniform. In fact, the leading digits of numbers from various sources such as atomic weights, baseball statistics, numbers in periodicals and values of mathematical functions or sequences seemed biased towards lower values; for instance, a leading digit of $1$ occurred about $30\%$ of the time, while a leading digit of $9$ occurred less than $5\%$ of the time. We now say a data set satisfies Benford's law (base $B$) if the probability of a first digit base $B$ of $d$ is $\log_B(1  + 1/d)$, or more generally the probability that the significand\footnote{If $x > 0$ we may write $x = S_B(x) 10^{k(x)}$, where $S_B(x) \in [1, B)$ is the significand and $k(x) \in \mathbb{Z}$ is the exponent.} is at most $s$ is $\log_B(s)$. Benford's law has applications in disciplines ranging from accounting (where it is used to detect fraud) to zoology and population growth, and many areas between. While this bias is often initially surprising, it is actually very natural as Benford's law is equivalent to the logarithms of the set being equidistributed modulo 1. For more on Benford's law see \cite{Hi1,Hi2,MT-B,Rai}, as well as \cite{Miller} for a compilation of articles on its theory and applications.

Obviously, we would not be discussing Benford's law if it had no connection to the Fibonacci numbers. A fascinating result, originally published in \cite{BD} (see also \cite{MT-B,Washington}), states that the Fibonacci numbers follow Benford's law of digit bias.\footnote{The main idea of the proof is to note that $\log_{10}\left(\frac{1+\sqrt{5}}{2}\right)$ is irrational, and then use Weyl's criterion and Binet's formula to show the logarithms of the Fibonacci numbers converge to being equidistributed modulo 1.} There are many questions that may be asked concerning the connection between the Fibonacci numbers and Benford's law. This research was motivated by the study of the distribution of leading digits of Fibonacci summands in Zeckendorf decompositions. Briefly, our main result is that the distribution of leading digits of summands in Zeckendorf decompositions converges to Benford's law. Our result is more universal, and in fact holds for special sequences with density. We first set some notation, and then precisely state our results.

\subsection{Preliminaries}\hspace*{\fill} \\

Let $S \subset \{F_i\}_{i = 1}^{\infty}$, and let $q(S, n)$ be the density of $S$ over the Fibonacci numbers in the interval $[0, F_n]$. That is,
\begin{align}
q(S,n) \ = \  \frac{\#\{F_{i} \in S \ : \ 1 \leq i \leq n\}}{n}.
\end{align}
When $\lim_{n \rightarrow \infty}{q(S,n)}$ exists, we define the \emph{asymptotic density} $q(S)$ as
\begin{align}
q(S) \ := \ \lim_{n \rightarrow \infty}{q(S,n)}.
\end{align}

For the sake of completeness, we define a mapping between the positive integers and their Zeckendorf decompositions. We first note that a \emph{legal} Zeckendorf decomposition is the unique decomposition of a number into non-adjacent Fibonacci numbers.

\begin{defn}
Let $m \in \N$. The function $\text{ZD}$ injectively maps each $m \in \N$ to the set of its Zeckendorf summands. Conversely, $\text{ZD}^{-1}$ injectively maps each legal set of Zeckendorf summands to the positive integer that set represents.
\end{defn}

For example, $\text{ZD}(10) = \{2,8\}$ and $\text{ZD}^{-1}(\{8,34\}) = 42$; however, $\text{ZD}^{-1}(\{8,13\})$ is undefined, as $21 = 8 + 13$ is not a legal Zeckendorf decomposition.

Let $m \in \N$ be chosen uniformly at random from the interval $[0, F_{n+1})$. We define two useful random variables:
\begin{align}
X_n(m) \ := \ \# \text{ZD}(m), \ \ \ \ \ Y_n(m) \ := \ \# \text{ZD}(m)  \cap S.
\end{align}
In our main result, we show that the density of $S$ in a typical Zeckendorf decomposition is asymptotic to the density of $S$ in the set of Fibonacci numbers.


\begin{thm}[Density Theorem for Zeckendorf Decompositions]\label{thm:DensityThm1} Let $S \subset \{ F_i \}_{i = 1}^{\infty}$ with asymptotic density $q(S)$ in the Fibonacci numbers. For $m \in \N$ chosen uniformly at random from the interval $[0, F_{n + 1})$, let $X_n(m)$ and $Y_n(m)$ be defined as above.  Then for any $\e>0$, we have with probability $1+o(1)$ that
\begin{align}
\left|\frac{Y_n(m)}{X_n(m)} \ - \ q(S)\right| \ < \ \e.
\end{align}
\end{thm}

We now define a method of constructing a random Zeckendorf decomposition, which plays a central role in our proofs. Essentially, we want to select a random subset of the Fibonacci numbers which satisfy the criterion of being a legal Zeckendorf decomposition. We fix a probability $p \in (0,1)$ and let $A_{n}(p)$ be a random subset of Fibonacci numbers at most $F_n$. Let $A_{0}(p) = \emptyset$, and define $A_{n}(p)$ recursively for $n>0$ as follows. We set
\begin{align} \label{eq:RandomProcess}
\threecase{A_{n}(p) \ = \ }{A_{n - 1}(p)}{if $F_{n - 1} \ \in \ A_{n - 1}(p)$}{A_{n - 1}(p) \cup F_n}{with probability $p$ if  $F_{n - 1} \ \notin \ A_{n - 1}(p)$}{A_{n - 1}(p)}{otherwise,}
\end{align} and define
\begin{align}
A(p) \ := \ \bigcup\limits_{n}{A_{n}(p)}.
\end{align}
This random process leads to the following result.

\begin{thm}[Density Theorem for Random Decompositions]\label{thm:DensityThm2}
Let $S \subset \{F_i\}_{i = 1}^{\infty}$ have asymptotic density $q(S)$ over the Fibonacci numbers. Then, with probability $1$, $S \cap A(p)$ has asymptotic density $q(S)$ in $A(p)$.
\end{thm}

We use Theorem \ref{thm:DensityThm2} with the clever choice of probability of $p=1/\varphi^2$ to prove Theorem \ref{thm:DensityThm1}. The reason for this choice is that this random Zeckendorf decomposition is similar to the Zeckendorf decomposition of an integer chosen uniformly at random.

We now describe some situations where Theorem \ref{thm:DensityThm2} applies. There are many interesting situations where $S \subset \{F_i\}_{i = 1}^{\infty}$ has a limiting density over the Fibonacci numbers. As the Fibonacci numbers follow Benford's law, the set $S_d$ of Fibonacci number with a fixed leading digit $1 \leq d \leq 9$ has asymptotic density $q(S_d) = \log\left({1 + 1/d}\right)$ in the Fibonacci numbers. By an extension of Benford's law, the Fibonacci numbers in which a finite amount of leading digits are fixed also have asymptotic density over the Fibonacci numbers. Conversely, we could fix a finite set of digits at the right and obtain similar results. For example, if we look at the Fibonacci numbers modulo $2$ we get $1, 0, 1, 1, 0, 1, 1, 0, \dots$; thus in the limit one-third of the Fibonacci numbers are even, and the asymptotic density exists. These arguments immediately imply Benford behavior of the Zeckendorf decompositions.

\begin{cor}[Benford Behavior in Zeckendorf Decompositions]\label{cor:zeckdecomparebenf} Fix positive integers $D$ and $B$, and let \be \mathcal{D}_D \ := \ \{(d_1, \dots, d_D): d_1 \ge 1, d_i \in \{0, 1, \dots, B-1\}\}; \ee to each $(d_1,\dots,d_D) \in \mathcal{D}_D$ we associate the set $S_{d_1,\dots,d_D}$ of Fibonacci numbers whose significand starts $d_1.d_2d_3\cdots d_D$. With probability $1$, for each $(d_1,\dots,d_D)$ we have $S_{d_1,\dots,d_D} \cap A(p)$ equals $\log_B (d_1.d_2d_3\cdots d_D)$, and thus with probability $1$ Benford's law holds.
\end{cor}

\begin{proof} As $D$ is fixed and finite, there are only finitely many starting blocks for significands in $\mathcal{D}_D$. By Theorem \ref{thm:DensityThm2} for each of these $S_{d_1,\dots,d_D} \cap S(p)$ equals the corresponding Benford probability; as the intersection of finitely many events that each happen with probability $1$ happens with probability $1$, we see that with probability $1$, all the significands of length $D$ happen with the correct probability. Sending $D\to\infty$ yields the desired Benford behavior. \end{proof}

As a check of our Benfordness results, we performed two simple experiments. The first was an exhaustive search of all $m \in [F_{25}, F_{26}) = [121393, 196418)$. We performed a chi-square goodness of fit test on the distribution of first digits of summands for each $m$ and Benford's law. There are eight degrees of freedom, and 99.74\% of the time our chi-square values were below the 95\% confidence threshold of 15.51, and 99.99\% of the time they were below the 99\% confidence threshold of 20.09. We then randomly chose a number in $[10^{60000}, 10^{60001})$, and found a chi-square value of 8.749. See Figure \ref{fig:chiindividualplot} for a comparison between the observed digit frequencies and Benford's law.

\begin{figure}
\begin{center}
\scalebox{.8}{\includegraphics{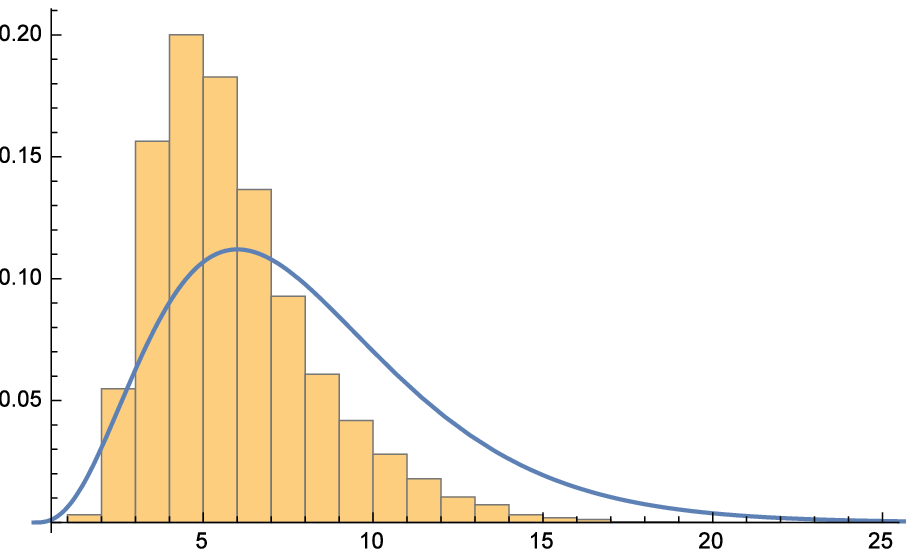}} \ \
\scalebox{.8}{\includegraphics{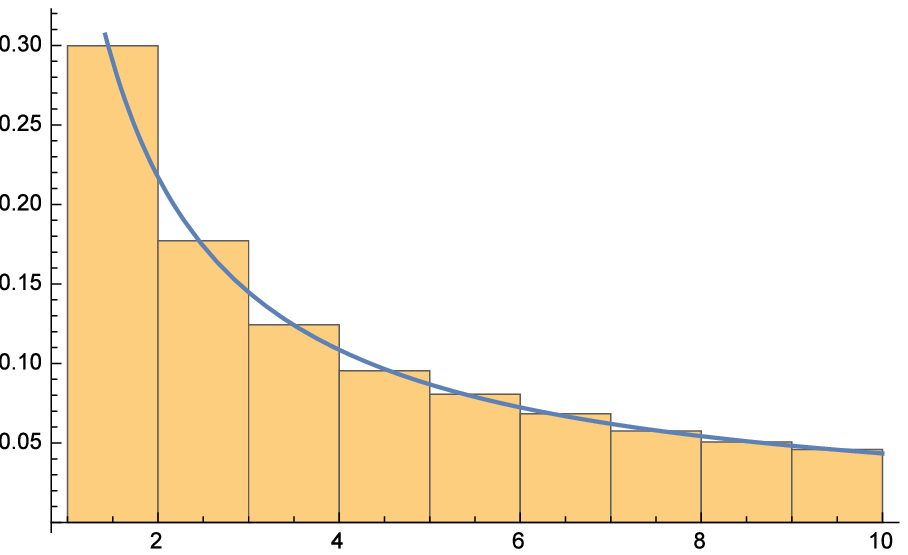}}
\caption{\label{fig:chiindividualplot}  Comparison of the frequencies of leading digits in Zeckendorf decompositions and Benford's law. Left: All integers in $[F_{25}, F_{26})$ (solid curve is a chi-square distribution with 8 degrees of freedom). Right: A large random integer, approximately $7.94 \cdot 10^{60000}$ (solid curve is $1/(x\log(10))$, the Benford density). }
\end{center}
\end{figure}

To prove our main results we first state and prove some lemmas about random legal decompositions. The key observation is that for an appropriate choice of $p$, the set $A(p)$ derived from the random process defined in \eqref{eq:RandomProcess} acts similarly to the Zeckendorf decomposition of a randomly chosen integer $m \in [0, F_{n+1})$. Theorem \ref{thm:DensityThm1} thus becomes a consequence Theorem \ref{thm:DensityThm2}, which we prove through Chebyshev's inequality.

\section{Proof of Theorem \ref{thm:DensityThm1}}\label{sec:benfbehrandzeck}


In this section, we assume the validity of Theorem \ref{thm:DensityThm2} in order to prove Theorem \ref{thm:DensityThm1}. The proof of Theorem \ref{thm:DensityThm2} is given in \S\ref{sec:benfbehaviorrandzeck}. We begin with a useful lemma on the probability our random process $A(p)$ equals $m$; interestingly, we find that $m \in [0, F_{n+1})$ are almost uniformly chosen.

\begin{lem}\label{lem:randprocfindingprobtwocases}
With $A_{n}(p)$ defined as in \eqref{eq:RandomProcess}, $\text{ZD}^{-1}(A_{n}(p)) \in [0, F_{n+1})$ is a random variable. For a fixed integer $m \in [0,F_{n+1})$ with the Zeckendorf decomposition $m = F_{a_1} + F_{a_2} + \cdots + F_{a_k}$, where $k \in \N, \ \ 1 \leq a_1, \ \ a_1 + 1 < a_2, \ \ \ldots, \ \ a_{k-1} + 1 < a_k$, we have
\begin{equation}\label{eq:lemmaexpforzdinverseprobk}
\twocase{\prob{\text{ZD}^{-1} \left(A_{n}(p) \right)  =  m} \ = \ }{p^k(1-p)^{n-2k}}{if $m \in [0, F_n)$}{p^k(1-p)^{n-2k+1}}{if $m \in [F_n, F_{n+1})$.}
\end{equation}
\end{lem}

\begin{proof} With probability $(1 - p)^{a_1 - 1}p$, $F_{a_1}$ is the smallest element of $A_{n}(p)$. For $j \in \Z$, suppose that $F_{a_1}, F_{a_2}, \ldots, F_{a_{j-1}}$ be the $j-1$ smallest elements of $A_{n}(p)$. With probability $(1 - p)^{a_j - a_{j-1} - 2}p$, $F_{a_j}$ is the next smallest element of $A_{n}(p)$; the reason we have a -2 in the exponent is that once we select $F_{a_1}$ we cannot have $F_{a_1+1}$, and thus there are $a_2-a_1-2$ Fibonacci numbers between $F_{a_1+1}$ and $F_{a_2-1}$ which we could have selected (but did not). Continuing,  we find $\text{ZD}^{-1} \left(A_{n}(p) \right) = m$ if and only if the $k$ smallest elements of $A_{n}(p)$ are $F_{a_1}, F_{a_2}, \ldots, F_{a_k}$ and $F_{j} \notin A_{n}(p)$ for $j > a_k$; note if $a_k=n$ then we are done determining if we have or do not have summands, while if $a_k < n$ we must elect not to have $F_{a_k+1}, \dots, F_n$ and thus need another $n-a_k-1$ factors of $1-p$. Then, by these calculations, $\text{ZD}^{-1} \left(A_{n}(p) \right) = m$ with probability
\begin{align} \label{eq:ProbabilityOfRandomZD}
\prob{\text{ZD}^{-1} \left(A_{n}(p) \right) \ = \ m} \ = \ (1-p)^{a_1-1}p \left(\prod_{j=2}^k{(1-p)^{a_j-a_{j-1}-2}p} \right) (1-p)^{n-a_k-\delta_k},  \end{align} where $\delta_k = 1$ if $a_k < n$ and $1$ if $a_k=n$. The first case happens when $m\in [0, F_n)$ and the second when $m \in [F_n, F_{n+1})$; \eqref{eq:lemmaexpforzdinverseprobk} now follows from simple algebra.
\end{proof}

The key idea in proving Theorem \ref{thm:DensityThm1} is to consider the special case of $p = 1/\varphi^2$ in Lemma \ref{lem:randprocfindingprobtwocases}, where $\varphi := \frac{1 + \sqrt{5}}{2}$ is the golden mean.\footnote{For us, the importance of $\varphi$ is that it is the largest root of the characteristic polynomial for the Fibonacci recurrence, and by Binet's formula it governs the growth of the sequence.} The reason this is an exceptionally useful choice is that initially the probability of choosing $m$ in our random process $A(p)$ depends on the \emph{number} of summands of $m$; however, for $p = 1/\varphi^2$ we have $p^k (1-p)^{-2k} = 1$. Thus in this case, for $m$ an integer in $[0, F_{n+1})$ we see that \eqref{eq:ProbabilityOfRandomZD} reduces to
\begin{align}\label{eq:zdinverseprobtwocases}
\twocase{\prob{\text{ZD}^{-1} \left(A_{n}(\varphi^{-2}) \right) = m} \ = \ }{\varphi^{-n}}{if $m\in [0,F_n)$}{\varphi^{-(n + 1)}}{if $m\in [F_n,F_{n+1})$.}
\end{align}
Note this is nearly independent of $m$; all that matters is whether or not it is larger than $F_n$. The desired result follows from straightforward algebra.\footnote{As a quick check, note $F_n \varphi^{-n} + (F_{n+1}-F_n) \varphi^{-(n+1)} = 1$, as required for a probability.}

We now are ready to prove Theorem \ref{thm:DensityThm1}.

\begin{proof}[Proof of Theorem \ref{thm:DensityThm1}]
For a fixed $\varepsilon > 0$, let
\begin{align}
E(n,\varepsilon) \ &:= \ \left\{m\in \Z \ \cap \ [0,F_{n+1}) \ : \ \left|\frac{Y_n(m)}{X_n(m)} \ - \ q(S) \right| \ \geq \ \varepsilon \right\}.
\end{align}

By Theorem \ref{thm:DensityThm2}, for $m$ chosen uniformly at random from the integers in $[0,F_{n+1})$, we have
\begin{align}
\prob{m \in E(n,\varepsilon)} & \ = \ \sum_{x \in E(n,\varepsilon)}{\frac{1}{F_{n+1}}} \nonumber\\
& \ = \ O\left(\sum_{x\in E(n,\varepsilon)}{\prob{\text{ZD}^{-1} \left(A_{n}(\varphi^{-2}) \right) \ = \ x}}\right) \nonumber\\
& \ = \ O\left(\prob{\text{ZD}^{-1} \left(A_{n}(\varphi^{-2}) \right) \in E(n,\varepsilon) }\right)  \ = \ o(1).
\end{align}
We conclude that that $\left|\frac{Y_n(m)}{X_n(m)}-q(S)\right| < \varepsilon$ with probability $1+o(1)$.
\end{proof}


\section{Proof of Theorem \ref{thm:DensityThm2}}\label{sec:benfbehaviorrandzeck}

In this section, we prove Theorem \ref{thm:DensityThm2}. We first prove some useful lemmas.

\begin{lem}
Let $A(p) \subset \{ F_n \}_{n = 1}^{\infty}$ be constructed as in \eqref{eq:RandomProcess} with probability parameter $p \in (0, 1)$. Then
\begin{align}
{\rm Prob}\left(F_k \in A(p) \right) \ &= \ \frac{p}{p + 1} + O(p^k).
\end{align}
\end{lem}

\begin{proof}
By conditioning on whether $F_{k - 2} \in A(p)$, we obtain a recurrence relation:\footnote{We can also give a simple heuristic suggesting the main term of the answer. For $k$ large, the probability $F_k$ occurs should roughly be the same as the probability that $F_{k-1}$ is used; call this $x$. Then $x \approx (1-x)p$ (to have $F_k$ we must first not have taken $F_{k-1}$, and then once this happens we choose $F_k$ with probability $p$), which implies $x \approx p/(1+p)$ as claimed.}
\begin{align}
{\rm Prob}\left(F_k \in A(p) \right) &\ = \  {\rm Prob}\left(F_k \in A(p) \ | \ F_{k-2} \in A(p) \right) \cdot {\rm Prob}\left(F_{k-2} \in A(p) \right)\nonumber\\ & \ \ \ \ \ \ \  +\  {\rm Prob}\left(F_k \in A(p) \ | \ F_{k-2} \notin A(p) \right) \cdot {\rm Prob}\left(F_{k-2} \notin A(p) \right) \nonumber\\
&\ = \  p \cdot {\rm Prob}\left(F_{k-2} \in A(p) \right) + p (1 - p) \cdot {\rm Prob}\left(F_{k-2} \notin A(p) \right) \nonumber\\
&\ = \  p^2 \cdot {\rm Prob}\left(F_{k-2} \in A(p) \right) + p - p^2.
\end{align}
As ${\rm Prob}\left(F_1 \in A(p) \right) \ = \  p$  and ${\rm Prob}\left(F_2 \in A(p) \right) \ = \  (1-p)p \ = \ p - p^2$, we have
\begin{align}
{\rm Prob}\left(F_k \in A(p) \right) &= \left( {\rm Prob}\left(F_{1} \in A(p) \right)\right)^2 \cdot {\rm Prob}\left(F_{k-2} \in A(p) \right) + {\rm Prob}\left(F_2 \in A(p) \right).
\end{align}
It is easy to show by induction that for all $k$,
\begin{align}
{\rm Prob}\left(F_k \in A(p) \right) \ = \  \sum\limits_{j = 1}^{k}{(-1)^{j + 1} p^j} \ = \ \frac{p}{1 + p} + O(p^k),
\end{align} completing the proof.
\end{proof}

\begin{lem} \label{lem:ExpVarW}
Let $W_n$ be the random variable defined by $W_n := \#A_{n}(p)$. Then
\begin{align}
\mathbb{E}[W_n] \ = \ \frac{n p}{1 + p} + O(1) \ \ \ \  {\rm and} \ \ \ \  {\rm Var}(W_n) \ &= \ O(n).
\end{align}
\end{lem}

\begin{proof}
Define the indicator function $\chi(F_k)$ for $k \in \N$ by
\begin{align}
\twocase{\chi(F_k) \ := \ }{1}{if $F_k \in A(p)$}{0}{if $F_k \notin A(p)$.}
\end{align}

We note that $W_n = \sum_{k = 1}^{n} {\chi(F_k)}$ and by linearity of expectation have
\begin{align}
\mathbb{E}[W_n] &\ = \  \sum\limits_{k = 1}^{n} {\mathbb{E}[\chi(F_k)]} \nonumber\\
&\ = \  \sum\limits_{k = 1}^{n} \prob{F_k \in A(p)} \nonumber\\
&\ = \  \sum\limits_{k = 1}^{n}\left(\frac{p}{1 + p} + O(p^k)\right) \nonumber\\
&\ = \  \frac{n p}{1 + p} + O(1).
\end{align}

To find the variance we use that it equals $\E[W_n^2] - \E[W_n]^2$. Without loss of generality, when we expand below we may assume $i \le j$ and double the contribution of certain terms. As we cannot have $F_i$ and $F_{i+1}$, there are dependencies. While we could determine the variance exactly with a bit more work, for our applications we only need to bound its order of magnitude.
\begin{align}
\mathbb{E}[W_n^2] &\ = \  \E\left[ \left( \sum\limits_{k = 1}^{n}{\chi(F_k)} \right)^2 \right] \nonumber\\
&\ = \  \E\left[ \sum\limits_{i, j \leq n}{\chi(F_i) \cdot \chi(F_j)} \right] \nonumber\\
&\ = \  \sum\limits_{i, j \leq n}{\mathbb{E}[\chi(F_i) \cdot \chi(F_j)]} \nonumber\\
&\ = \ \sum_{i,j\leq n}{\p{F_i\in A(p)}\p{F_j\in A(p)|F_i\in A(p)}} \nonumber\\
&\ = \ \sum_{i\leq n}{\p{F_i\in A(p)}}+2\sum_{i+2\leq j\leq n}{\p{F_i\in A(p)}\p{F_j\in A(p)|F_i\in A(p)}} \nonumber\\
&\ = \ O(n)+2\sum_{i+2\leq j\leq n}{\p{F_i\in A(p)}\p{F_{j-i-1}\in A(p)}} \nonumber\\
&\ = \ O(n)+2\sum_{i+2\leq j\leq n}{\left(\frac{p}{1+p}\right)^2\left(1+O\left(p^{\min(i,j-i)}\right)\right)} \nonumber\\
&\ \leq \ O(n)+\left(\frac{np}{1+p}\right)^2+O\left(\sum_{i+2\leq j\leq n}{p^{\min(i,j-i)}}\right).
\end{align}

For a fixed $k = 1, 2, \ldots, n-1$, there are less than $n$ pairs $(i,j)$ with $k = i < j - i$ and $i + 2\leq j \leq n$.  Similarly, there are less than $n$ pairs $(i,j)$ with $k = i - j \leq i$, $i + 2 \leq j \leq n$.  Therefore, there are less than $2n$ pairs $(i,j)$ for which $\min(i,j-i)=k$.  Thus
\begin{align}
\sum_{i+2\leq j\leq n}{p^{\min(i,j-i)}} \ < \ 2n\sum_{k=1}^{n-1}{p^k}  \ = \ O(n),
\end{align}
and therefore
\begin{align}
\mathbb{E}[W_n^2]  \ = \ \left(\frac{np}{1+p}\right)^2+O(n)  \ = \ \mathbb{E}[W_n]^2+O(n).
\end{align}
We conclude that
\begin{align}
\text{Var}(W_n) \ = \ O(n),
\end{align}
completing the proof.
\end{proof}


\begin{cor} \label{cor:RandomVarW} Let $W_n$ be the random variable defined by $W_n := \#A_{n}(p)$. With probability $1+o(1)$,
\begin{align}
\left|W_n-\frac{np}{1+p}\right| \ <\ n^{2/3}.
\end{align}
\end{cor}

\begin{proof}
This follows immediately from Chebyshev's inequality, as for large $n$ we have
\begin{align}
\p{\left|W_n - \frac{np}{1+p} \right| \geq n^{2/3}} &\ \leq \ \p{\left|W_n - \mathbb{E}[W_n] \right| \geq \frac{\mathbb{E}[W_n]}{n^{4/9}}} \nonumber\\
&\ \leq\  \frac{{\rm Var}(W_n) n^{8/9}}{\mathbb{E}[W_n]^2} \ = \ o(1).
\end{align}
\end{proof}


\begin{lem}
Let $S \ \subset \ \{ F_n \}_{n = 1}^{\infty}$ with asymptotic density $q(S)$ in the Fibonacci numbers. Let $Z_n$ be the random variable defined by $Z_n \ := \ \#{A_n(p) \cap S}$. Then
\begin{align}
\mathbb{E}[Z_n] \ &= \ \frac{n p q(S)}{1 \ + \ p} \ + \ o(n) \nonumber\\
{\rm Var}(Z_n) \ &= \ o(n^2).
\end{align}
\end{lem}

\begin{proof}
Define the indicator function $\psi(F_k)$ for $k \in \N$ by
\begin{align}
\twocase{\psi(F_k) \ =\ }{1}{if $F_k \in S$}{0}{if $F_k \notin S$.}
\end{align}
Then we have
\begin{align}
\mathbb{E}[Z_n] \ &= \ \sum_{k=1}^n{\psi(F_k) \prob{F_k\in A(p)}} \nonumber\\
& \ = \ \sum_{k=1}^n{\psi(F_k) \left(\frac{p}{1+p}+O(p^k)\right)} \nonumber\\
& \ = \ O(1)+\frac{p}{1+p}\sum_{k=1}^n{\psi(F_k)} \nonumber\\
& \ = \ \frac{n p q(S)}{1+p}+o(n)
\end{align}
since $\lim_{n\to\infty}{q(S,n)}=q(S)$.

Similarly to the calculation in Lemma \ref{lem:ExpVarW}, we compute
\begin{align}
\mathbb{E}[Z_n^2] & \ = \ \sum_{i,j\leq n}{\psi(F_i) \psi(F_j) \p{F_i\in A(p)}\p{F_j\in A(p)|F_i\in A(p)}} \nonumber\\
& \ = \ O(n)+2\sum_{i+2\leq j\leq n}{\psi(F_i) \psi(F_j) \p{F_i\in A(p)}\p{F_{j-i-1}\in A(p)}} \nonumber\\
& \ = \ O(n)+2\sum_{i+2\leq j\leq n}{\psi(F_i) \psi(F_j) \left(\frac{p}{1+p}\right)^2\left(1+O\left(p^{\min(i,j-i)}\right)\right)} \nonumber\\
& \ = \ O(n)+2\left(\frac{p}{1+p}\right)^2\sum_{i+2\leq j\leq n}{\psi(F_i) \psi(F_j)} \nonumber\\
& \ = \ o(n^2)+\left(\frac{npq(S)}{1+p}\right)^2.
\end{align}

In the calculation above, the only difficulty is in the second to last line, where we argued that the main term of the $i$ and $j$ double sum was $n^2 q(S)^2/2$. To see this, note by symmetry that up to contributions of size $O(n)$ we can remove the restrictions on $i$ and $j$ (and thus have each range from 1 to $n$) if we then take half of the resulting sum. Thus, the restricted double sum becomes $\frac12 \left(\sum_{i \le n} \psi(F_i)\right) \left(\sum_{j \le n} \psi(F_j)\right)$, which as $n\to\infty$ converges to $\frac12 q(S)n \cdot q(S)n$ (up to an error of size $O(n)$, of course). Therefore, we have
\begin{align}
\text{Var}(Z_n)  \ = \ \mathbb{E}[Z_n^2]-\mathbb{E}[Z_n]^2  \ = \ o(n^2),
\end{align}
which completes the proof.
\end{proof}


\begin{cor} \label{cor:RandomVarZ}
Let $Z_n$ be the random variable defined by $Z_n := \# A_n(p) \cap S$, and let $g(n)=n^{1/2}{\rm Var}(Z_n)^{-1/4}$. Then
\begin{align}
\p{|Z_n-\mathbb{E}[Z_n]|>\frac{\mathbb{E}[Z_n]}{g(n)}} &\ \leq\ \frac{{\rm Var}(Z_n)g(n)^2}{\mathbb{E}[Z_n]^2}  \ = \ o(1).
\end{align}
\end{cor}

\begin{proof}
The proof follows immediately by Chebyshev's inequality and the order of magnitude of the various quantities.
\end{proof}

Armed with the above results, we can now prove our main theorem.

\begin{proof}[Proof of Theorem \ref{thm:DensityThm2}]
Let
\begin{align}
& e_1(n) \ = \ n^{-1/3}, \nonumber\\
& e_2(n) \ = \ \frac{1}{n}\left(\frac{\mathbb{E}[Z_n]}{g(n)}+\left|E[Z_n]-\frac{npq(S)}{1+p}\right| \right).
\end{align}

Note that both are of order $o(1)$.  We combine Corollaries \ref{cor:RandomVarW} and \ref{cor:RandomVarZ} to see that with probability $1+o(1)$ we have
\begin{align}
& Z_n \ \leq \ \frac{npq(S)}{1 + p}(1+e_2(n)),  \nonumber\\
& W_n \ \geq \ \frac{np}{1 + p}(1-e_1(n)).
\end{align}
Therefore, for any $\e>0$ we have with probability $1$ that
\begin{align}
\lim_{n\to\infty}{\frac{Z_n}{W_n}} & \ \leq \ \lim_{n\to\infty}{\frac{q(S)(1+e_2(n))}{1-e_1(n)}}  \ = \ q(S).
\end{align}
A similar argument gives $q(S)$ as a lower bound for $\lim_{n\to\infty} Z_n/W_n$, and thus with probability $1$
\begin{align}
\lim_{n\to\infty}{\frac{Z_n}{W_n}}\ = \ q(S),
\end{align} as desired.
\end{proof}


\section{Conclusion and Future Work}\label{sec:concfuturework}

We were able to handle the behavior of almost all Zeckendorf decompositions by finding a correspondence between these and a special random process, replacing the deterministic behavior for each $m \in [0, F_n)$ with random behavior which is easier to analyze. The key observation was that this correspondence held when choosing $p = 1/\varphi^2$. This allowed us to prove not just Benford behavior for the leading digits of summands in almost all Zeckendorf decompositions, but also similar results for other sequences with density.

In future work we plan on revisiting these problems for more general recurrences, where there is an extensive literature (see among others \cite{Al,Day,DDKMMV,DDKMV,DG,FGNPT,GT,GTNP,Ke,Len,MW1,MW2,Ste1,Ste2}). Similar to other papers in the field (for example, \cite{KKMW} versus \cite{MW1}, or \cite{CFHMN1} versus \cite{CFHMN2}), the arguments are often easier for the Fibonacci numbers, as we have simpler and more explicit formulas at our disposal.


\ \\

\end{document}